\documentclass[10pt,reqno]{amsart}
\usepackage{amsmath}
\usepackage{amsthm}
\usepackage{amssymb}
\usepackage{amsfonts}
\usepackage{latexsym, multicol}
\usepackage[backref=page]{hyperref}

\usepackage{pb-diagram}

\renewcommand{\phi}{\varphi}

\newcommand{\tr}{\operatorname{tr}}
\newcommand{\C}{\mathbb{C}}

\newcommand{\M}{{\bf M}}
\newcommand{\SU}{{\bf SU}}
\renewcommand{\O}{{\bf O}}

\newcommand{\R}{\mathbb{R}}

\newcommand{\norm}[1]{\left\| #1 \right\|}
\newcommand{\inner}[1]{\langle #1 \rangle}

\newcommand{\diag}{\operatorname{diag}}

\newcommand{\megamatrix}[9]{\begin{pmatrix} #1 & #2 & #3 \\ #4 & #5 & #6 \\ #7 & #8 & #9\end{pmatrix}  }
\newcommand{\ultramatrix}[6]{\begin{pmatrix} 0 & #1 & #2 & #3 \\ 0 & 0 & #4 & #5 \\ 0 & 0 & 0 & #6 \\ 0 & 0 & 0 & 0\end{pmatrix}  }

\newtheorem{Corollary}{Corollary}
\newtheorem{Theorem}{Theorem}

\newtheorem{Lemma}{Lemma}
\theoremstyle{definition}

\newtheorem{Example}{Example}

\allowdisplaybreaks
\begin{document}
\title[Unitary equivalence to a complex symmetric matrix]{Unitary equivalence to a complex symmetric matrix:  low dimensions}
\author[S.R.~Garcia]{Stephan Ramon Garcia}

\author[D.~Poore]{Daniel E. Poore}

	\address{Department of Mathematics\\
	Pomona College\\
	610 North College Avenue\\
	Claremont, California\\
	91711}
	\email{Stephan.Garcia@pomona.edu}
	\urladdr{http://pages.pomona.edu/\textasciitilde sg064747}

	\author[J.E.~Tener]{James E. Tener}
	\address{   Department of Mathematics\\
		University of California\\Berkeley, CA\\ 94721}
	\email{jtener@math.berkeley.edu}
	\urladdr{http://math.berkeley.edu/\textasciitilde jtener}

	\thanks{The first and second authors were partially supported by NSF Grant DMS-1001614.  The third author
	was partially supported by a NSF Graduate Research Fellowship.}

	
	\begin{abstract}
		A matrix $T \in \M_n(\C)$ is \emph{UECSM} if it is 
		unitarily equivalent to a complex symmetric (i.e., self-transpose) matrix.
		We develop several techniques for studying this property in dimensions three and four.
		Among other things, we completely characterize $4 \times 4$ nilpotent matrices which
		are UECSM and we settle an open problem which has lingered in the $3 \times 3$ case.
		We conclude with a discussion concerning a crucial difference which makes dimension
		three so different from dimensions four and above	
	\end{abstract}

\bibliographystyle{plain}


\maketitle

\section{Introduction}

	Following \cite{Tener}, we say that a matrix $T \in \M_n(\C)$ is \emph{UECSM} if it is 
	unitarily equivalent to a complex symmetric (i.e., self-transpose) matrix.  
	Here we use the term \emph{unitarily equivalent} in the sense of operator theory:
	we say that two matrices $A$ and $B$ are unitarily equivalent if $A = UBU^*$ for some unitary matrix $U$.
	We denote this relationship by $A \cong B$.
	In contrast, the term \emph{unitarily similar} is frequently used in the matrix-theory literature.
	
	Since every square complex matrix is \emph{similar} to
	a complex symmetric matrix \cite[Thm.~4.4.9]{HJ} (see also \cite[Ex.~4]{CSOA} and \cite[Thm.~2.3]{CCO}),
	determining whether a given matrix is UECSM is sometimes difficult, although
	several numerical methods \cite{UECSMGC, UECSMMC, Tener} have recently emerged.
	To illustrate the subtlety of this problem, we remark that exactly one of the following matrices is UECSM (see
	Section \ref{SectionN4})
	\begin{equation}\label{eq-Stump}
		\ultramatrix{2}{9}{1}{0}{4}{7}\quad
		\ultramatrix{2}{9}{1}{0}{5}{7}\quad
		\ultramatrix{2}{9}{1}{0}{6}{7}\quad
		\ultramatrix{2}{9}{1}{0}{7}{7}.
	\end{equation}

	Let us briefly discuss our results.
	First, we adapt, from the three-dimensional to the four-dimensional setting,
	a highly successful method developed in \cite{ATTO} based upon the Pearcy-Sibirski{\u\i} trace
	criteria \cite{Pearcy, Sibirskii} (Section \ref{SectionTrace}).  This work depends crucially
	upon a recent breakthrough of Djokivi\'{c} \cite{Djokovic} in the study of Poincar\'e series.
	As a concrete example, we use our new criteria to completely characterize $4 \times 4$ nilpotent matrices which
	are UECSM (Section \ref{SectionN4}).  Following a somewhat different thread, we settle in the affirmative 
	a conjecture which has lingered in the $3 \times 3$ case for the last few years (Section \ref{SectionWAT}).
	Moreover, we also provide a theoretical explanation for the failure of this conjecture in dimensions four and above
	(Section \ref{SectionInsuffWAT}).  In particular, we are able to construct the counterexample \cite[Ex.~5]{UECSMGC}
	from scratch, as opposed to resorting to a brute-force random search.  	
	We conclude this note with a discussion concerning a crucial difference which makes dimension
	three so different from dimensions four and above (Section \ref{Section34}).
	\medskip
	
	\noindent\textbf{Acknowledgments}:  The first author wishes to thank Bernd Sturmfels
	for a recent helpful discussion about symbolic computation with complex variables.  The 
	authors would also like to thank the anonymous referee for a careful review of this paper and for
	suggesting a way to simplify the proof of Theorem \ref{TheoremWAT}.

\section{Trace Criteria}\label{SectionTrace}

	In 1968, Sibirski{\u\i} \cite{Sibirskii} refined a striking result of Pearcy \cite{Pearcy}
	and proved that $A,B \in \M_3(\C)$ are unitarily equivalent if and only if $\Phi(A) = \Phi(B)$ 
	where $\Phi:\M_3(\C)\to \C^7$ is the function defined by
	\begin{equation}\label{eq-Words}
		\Phi(X) = (\tr X, \, \tr  X^2,\, \tr X^3,\, \tr X^* X,\, \tr X^*X^2, \, \tr {X^*}^2 X^2, \, \tr X^* X^2 {X^*}^2X).
	\end{equation}
	Pearcy's original 1962 result included the words $X^*XX^*X$ and $X^*X^2X^*X$,
	which were later shown by Sibirksi{\u\i} to be redundant.  
	
	Recently, the first and third authors proved that for $n \leq 7$, a matrix 
	$T \in \M_n(\C)$ is UECSM if and only if $T \cong T^t$ and, moreover, that this result fails for $n \geq 8$ \cite{UET}.
	Consequently, $T \in \M_3(\C)$ is UECSM if and only if
	$\Phi(T) = \Phi(T^t)$.
	Fortunately, the first six traces in \eqref{eq-Words} are automatically
	equal for $X = T$ and $X = T^t$, whence $T$ is UECSM 
	if and only if $\tr X^* X^2 {X^*}^2X$ yields the same value for $X = T$ and $X = T^t$.
	Using standard properties of the trace, one sees that this is equivalent to 
	\begin{equation}\label{eq-TraceTest}
		\tr [T^*T(T^*T - TT^*)TT^*] = 0.
	\end{equation}
	In other words, $T \in \M_3(\C)$ is UECSM if and only if \eqref{eq-TraceTest} holds.	

	A simple extension of the Pearcy-Sibirski{\u\i} theorem to the $4 \times 4$ setting appeared hopeless for many years 
	until Djokovi\'c \cite[Thm.~4.4]{Djokovic} recently proved that $A,B \in \M_4(\C)$
	are unitarily equivalent if and only if $\tr w_i(A,A^*) = \tr w_i(B,B^*)$
	for $i = 1,2,\ldots,20$, where the words $w_i(x,y)$ are defined by
	\begin{multicols}{4}
		\begin{enumerate}\addtolength{\itemsep}{0.25\baselineskip}
			\item $x$
			\item $x^2$
			\item $xy$
			\item $x^3$
			\item $x^2 y$
			\columnbreak
			\item $x^4$
			\item $x^3 y $
			\item $x^2 y^2$
			\item $xyxy$
			\item $x^3 y^2$
			\columnbreak
			\item $x^2yx^2y$
			\item $x^2 y^2 xy$
			\item $y^2 x^2 y x$
			\item $x^3 y^2 xy$
			\item $x^3 y^2 x^2 y$
			\columnbreak
			\item $x^3 y^3 xy$
			\item $y^3 x^3 y x$
			\item $x^3 y x^2 yxy$
			\item $x^2 y^2 x y x^2 y$
			\item $x^3 y^3 x^2 y^2$.
		\end{enumerate}
	\end{multicols}
	In light of the fact that $T \in \M_4(\C)$ is UECSM if and only if $T \cong T^t$, it follows that
	$T$ is UECSM if and only if $\tr w_i(T,T^*) = \tr w_i(T^t,\overline{T})$
	for $i = 1,2,\ldots,20$.  Since a matrix and its transpose have the same trace,
	the preceding is equivalent to
	\begin{equation}\label{eq-TraceTransposeReverse}
		\tr w_i(T,T^*) = \tr \widetilde{w}_i(T,T^*),
	\end{equation}
	where $\widetilde{w}_i(x,y)$ is the reverse of $w_i(x,y)$ (e.g., $\widetilde{xy^2} = y^2 x$).
	Fortunately, the desired condition \eqref{eq-TraceTransposeReverse} holds automatically for $i = 1,2,\ldots, 11$.
	For instance, 
	\begin{equation*}
		\tr \widetilde{w}_{11}(T,T^*)
		= \tr T^*T^2 T^*T^2 
		= \tr T^2 T^* T^2 T^* 
		= \tr w_{11}(T,T^*).
	\end{equation*}
	Thus $T$ is UECSM if and only if \eqref{eq-TraceTransposeReverse} holds for the nine values $i = 12,13,\ldots,20$.
	However, we can do even better for we claim that \eqref{eq-TraceTransposeReverse} holds for $i = 12$ if and only if
	\eqref{eq-TraceTransposeReverse} holds for $i = 13$:
	\begin{align*}
		\tr w_{12}(T,T^*) = \tr \widetilde{w}_{12}(T,T^*)
		&\quad \Leftrightarrow\quad \tr T^2 {T^*}^2 TT^* = \tr T^* T {T^*}^2 T^2 \\
		&\quad \Leftrightarrow\quad  \tr  TT^* T^2 {T^*}^2 = \tr {T^*}^2 T^2 T^* T\\
		&\quad \Leftrightarrow\quad  \tr \widetilde{w}_{13}(T,T^*) = \tr w_{13}(T,T^*) .
	\end{align*}
	Similarly, \eqref{eq-TraceTransposeReverse} holds for $i = 16$ if and only if \eqref{eq-TraceTransposeReverse} holds for $i = 17$:
	\begin{align*}
		\tr w_{16}(T,T^*) = \tr \widetilde{w}_{16}(T,T^*)
		&\quad \Leftrightarrow\quad \tr T^3 {T^*}^3 TT^* = \tr T^*T {T^*}^3 T^3\\
		&\quad \Leftrightarrow\quad  \tr TT^* T^3 {T^*}^3 = \tr {T^*}^3 T^3 T^*T  \\
		&\quad \Leftrightarrow\quad  \tr \widetilde{w}_{17}(T,T^*) = \tr w_{17}(T,T^*) .
	\end{align*}
	Thus we need only consider the indices $i = 12,14,15,16,18,19,20$.
	Now observe that for $i = 20$ the desired condition $\tr w_{20}(T,T^*) = \tr \widetilde{w}_{20}(T,T^*)$ is equivalent to
	\begin{align*}
		\tr ( T^3 {T^*}^3 T^2 {T^*}^2 - {T^*}^2 T^2 {T^*}^3 T^3) = 0
		&\,\, \Leftrightarrow\,\, \tr ( T^3 {T^*}^3 T^2 {T^*}^2 -  T^2 {T^*}^3 T^3{T^*}^2 ) = 0 \\
		&\,\, \Leftrightarrow\,\, \tr [T^2(T{T^*}^3 - {T^*}^3 T) T^2 {T^*}^2]=0.
	\end{align*}
	Similar computations for $i = 12,14,15,16,18,19$ yield the following theorem:

	\begin{Theorem}\label{TheoremTT4}
		A matrix $T \in {\bf M}_4(\C)$ is UECSM if and only if the traces of the following
		seven matrices vanish:\vspace{-8pt}
		\begin{multicols}{2}
			\begin{enumerate}\addtolength{\itemsep}{0.25\baselineskip}
				\item $T (T {T^*}^2 - {T^*}^2 T)TT^*$,
				\item $T(T^2 {T^*}^2 - {T^*}^2 T^2)TT^*$,
				\item $T^2(T{T^*}^2 - {T^*}^2 T)T^2T^*$,
				\item $T(T^2 {T^*}^3 - {T^*}^3 T^2)TT^*$,
				\columnbreak
				\item $T [ (T^2 T^*)^2 - (T^* T^2)^2]TT^*$,
				\item $T^2 T^*(T^*T - TT^*)T^*T^2T^*$,
				\item $T^2(T{T^*}^3 - {T^*}^3 T) T^2 {T^*}^2$.
			\end{enumerate}
		\end{multicols}
	\end{Theorem}

	For the sake of convenience, we adopt the following notation.
	For $i=1,2,\ldots,7$, let $\Psi_i(T)$ denote the trace of the $i$th matrix listed in Theorem \ref{TheoremTT4} 
	and define a function $\Psi:\M_4(\C)\to\C^7$ by setting
	$\Psi(T) = ( \Psi_1(T), \Psi_2(T),\ldots, \Psi_7(T))$.  In light of Theorem \ref{TheoremTT4}, we see that
	$T \in \M_4(\C)$ is UECSM if and only if $\Psi(T) = 0$.

	\begin{Example}	
		In \cite{UECSMMC} it is observed that neither of the matrices
		\begin{equation*}
			T_1=
			\begin{pmatrix}
				1 &0 & 0 & 0 \\
				0 & 0 & 2 & 0  \\
				0 & 0 & 0 & 2 \\
				0 &0 & 0 & 0 \\
			\end{pmatrix},\qquad
			T_2=
			\begin{pmatrix}
				1 &0 & 0 & 0 \\
				0 & 0 & 1 & 0  \\
				0 & 0 & 0 & 2 \\
				0 &0 & 0 & 0 \\
			\end{pmatrix} ,
		\end{equation*}
		are susceptible to testing with \texttt{UECSMTest} \cite{Tener}, 
		\texttt{ModulusTest} \cite{UECSMMC}, or \texttt{StrongAngleTest} \cite{UECSMGC},
		although ad-hoc arguments can be employed.  
		Since $\Psi(T_1) = 0$ and $\Psi(T_2) = (-12,0,0,0,0,0,0)$, we conclude that
		$T_1$ is UECSM and that $T_2$ is not.
	\end{Example}

\section{Canonical forms: $4 \times 4$ nilpotent UECSMs}\label{SectionN4}
	As an application example of Theorem \ref{TheoremTT4} we completely characterize those
	$4 \times 4$ nilpotent matrices which are UECSM.  This is an illuminating exercise for several reasons.
	First of all, characterizing objects up to unitary equivalence is typically a difficult task
	and previous work has mostly been confined to the $3 \times 3$ case (e.g., \cite[Thms.~5.1, 5.2]{ATTO}, \cite[Sect.~4]{Tener}).
	Second, we encounter many families which can be independently proven to be UECSM
	based upon purely theoretical considerations (i.e., providing independent confirmation of our results).  Finally, we discover several
	interesting classes of matrices which are UECSM but which do not
	fall into any previously known class.

	In light of Schur's Theorem on unitary triangularization, we restrict
	our attention to matrices of the form
	\begin{equation}\label{eq-Nilpotent}
		T = \ultramatrix{a}{b}{c}{d}{e}{f}.
	\end{equation}
	Noting that $T^3$ has at most one nonzero
	entry, we first consider the fourth and seventh conditions in Theorem \ref{TheoremTT4} since we expect
	these traces to be simple when expanded symbolically.   Indeed, a computation
	reveals that
	\begin{align}
		\Psi_4(T) &= |a|^2 |d|^2 |f|^2 (|a|^2 +|b|^2 - |e|^2 - |f|^2), \label{eq-Phi4} \\
		\Psi_7(T) &= |a|^2 |d|^4 |f|^2 (|a|^2 - |f|^2). \label{eq-Phi7}
	\end{align}
	Since we require $\Psi(T) = 0$, 
	we examine several special cases.

\subsection{The case $d=0$}
	A few routine computations tell us that
	$\Psi_i(T) = 0$ for $i = 2,3,4,5,7$.  For $i = 1$ and $i = 6$ we have
	\begin{align*}
		\Psi_1(T) &= |ae+bf|^2 (|a|^2 + |b|^2 - |e|^2 - |f|^2) , \\
		\Psi_6(T) &= \overline{c}(ae + bf) \Psi_1(T),
	\end{align*}
	whence $\Psi(T) = 0$ if and only if either
	\begin{equation}\label{eq-SurpriseNilpotent}
		 ae+bf = 0 
	\end{equation}
	or
	\begin{equation}\label{eq-NormSame}
		|a|^2 + |b|^2 = |e|^2 + |f|^2. 
	\end{equation}
	The condition \eqref{eq-SurpriseNilpotent} has a simple interpretation, for if 
	$d=0$, then $T^2 = 0$ if and only if \eqref{eq-SurpriseNilpotent} holds.  Now recall that a matrix which is nilpotent 
	of order two is UECSM \cite[Cor.~4]{SNCSO}.
	
	On the other hand, condition \eqref{eq-NormSame} does not have an obvious
	theoretical interpretation.  We remark that the third matrix in \eqref{eq-Stump} is obtained
	by setting $a = 2$, $b = 9$, $c=1$, $d = 0$, $e = 6$, $f = 7$ and noting that
	$2^2 + 9^2 = 85 = 6^2 + 7^2$.  The remaining three matrices in \eqref{eq-Stump} are not UECSM
	since their entries do not satisfy \eqref{eq-NormSame}.
	
\subsection{The case $a=0$}\label{SubsectionA}
	In this case, let us write 
	\begin{equation*}
		T = 
		\begin{pmatrix}
			\begin{array}{c|c|c|c}
				0&0&b&c\\
				0&0&d&e\\
				0&0&0&f\\
				0&0&0&0
			\end{array}
		\end{pmatrix}.
	\end{equation*}
	Another calculation shows that $\Psi_i(T) = 0$ for $i = 2,3,4,5,7$ and that
	\begin{align}
		\Psi_1(T) 
		&= |f|^2 \Big[\, (\underbrace{|b|^2 + |d|^2}_{\norm{v_3}^2} )
			( \underbrace{ |b|^2 + |d|^2 - |c|^2 - |e|^2 - |f|^2 }_{ \norm{v_3}^2 - \norm{v_4}^2}) 
			+ |\underbrace{b \overline{c} + d \overline{e} }_{\inner{v_3, v_4}} |^2\, \Big], \label{eq-Option1}\\
		\Psi_6(T) &=  f (\underbrace{ b \overline{c} + d \overline{e} }_{ \inner{ v_3, v_4}}) \Psi_1(T), \label{eq-Option2}
	\end{align}
	where $v_1, v_2, v_3, v_4$ denote the columns of $T$.
	Depending upon whether $f = 0$ or not, there are two cases to consider.
	\begin{enumerate}\addtolength{\itemsep}{0.5\baselineskip}
		\item If $f = 0$, then $\Psi(T) = 0$ whence $T$ is UECSM.  This agrees with theory, since
			in this case
			\begin{equation*}
				T = 
				\left(
				\begin{array}{cc|cc}
					0 & 0 & b & c \\
					0 & 0 & d & e \\
					\hline
					0 & 0 & 0 & 0 \\
					0 & 0 & 0 & 0 
				\end{array}
				\right)
			\end{equation*}
			is nilpotent of order two and hence UECSM by \cite[Cor.~4]{SNCSO}.
		
		\item If $f \neq 0$, then according to \eqref{eq-Option1} and \eqref{eq-Option2} there are several possibilities.
			\medskip
			\begin{enumerate}\addtolength{\itemsep}{0.5\baselineskip}
				\item If $|b|^2 + |d|^2 = 0$, then $b=d=0$ and $\Psi(T) = 0$ whence $T$ is UECSM.  This agrees
					with the fact that every rank-one matrix is UECSM \cite[Cor.~5]{SNCSO}.
					
				\item If $(|b|^2+|d|^2)(|b|^2 + |d|^2 - |c|^2 - |e|^2 - |f|^2) + |b \overline{c} + d \overline{e}|^2 =0$,
					then $T$ is UECSM.
					In particular, observe that if $v_3$ and $v_4$ are orthogonal vectors 
					with the same norm, then $\Psi(T) = 0$.  This agrees
					with the observation that every partial isometry on $\C^4$ is UECSM
					\cite[Cor.~2]{CSPI}.  Otherwise we obtain matrices which are UECSM but 
					which do not lie in any previously understood class.
			\end{enumerate}
	\end{enumerate}
	
\subsection{The case $f=0$}
	In this case we have
	\begin{equation*}
		T^t \cong
		\begin{pmatrix}
			0 & 0 & e & c\\
			0 & 0 & d & b\\
			0 & 0 & 0 & a\\
			0 & 0 & 0 & 0
		\end{pmatrix}.
	\end{equation*}
	We therefore have the same results as Subsection \ref{SubsectionA}, after exchanging the roles of $a$ and $f$, and $b$ and $e$, respectively.

\subsection{The case $a,d,f \neq 0$}
	If $a,d,f \neq 0$, then it follows from \eqref{eq-Phi4} and \eqref{eq-Phi7} that the conditions
	$|a| = |f|$ and $|b| = |e|$ are necessary for $T$ to be UECSM.  In fact, we claim that these
	conditions are also sufficient.  Indeed, if $|a| = |f|$ and $|b| = |e|$, then upon conjugating $T$ by a diagonal unitary matrix we see that
	\begin{equation*}
		T \cong \ultramatrix{a}{b}{c}{d}{b}{a},
	\end{equation*}
	which is unitarily equivalent to its transpose via the symmetric unitary matrix
	\begin{equation*}
		U = 
		\begin{pmatrix}
			0&0&0&1\\
			0&0&1&0\\
			0&1&0&0\\
			1&0&0&0
		\end{pmatrix}.
	\end{equation*}
	Thus $T$ is UECSM whenever $a,d,f \neq 0$, $|a|=|f|$, and $|b| = |e|$.
	
	\bigskip
	
	The following theorem summarizes our findings:

	\begin{Theorem}\label{TheoremN4}
		The matrix
		\begin{equation*}	
			T = \ultramatrix{a}{b}{c}{d}{e}{f}
		\end{equation*}
		is UECSM if and only if at least one of the following occurs:\smallskip
		\begin{enumerate}\addtolength{\itemsep}{0.5\baselineskip}
			\item $d=0$ and $ae+bf = 0$,
			\item $d=0$ and $|a|^2+|b|^2 = |e|^2 + |f|^2$,
			\item $a=0$ and $f=0$,
			\item $a=0$ and $(|b|^2+|d|^2)(|b|^2 + |d|^2 -|c|^2 - |e|^2 - |f|^2) + | b \overline{c} + d \overline{e}|^2 = 0$,
			\item $f = 0$ and $(|d|^2+|e|^2)(|d|^2+|e|^2 - |a|^2 - |b|^2 - |c|^2) + | c \overline{e} + b \overline{d}|^2 = 0$,
			\item $|a| = |f|$ and $|b|=|e|$.
		\end{enumerate}
	\end{Theorem}

\section{An angle criterion in three dimensions}\label{SectionWAT}

	Suppose that $T \in \M_n(\C)$ has distinct eigenvalues $\lambda_1, \lambda_2, \ldots, \lambda_n$
	with corresponding normalized eigenvectors $x_1,x_2,\ldots, x_n$.  Let
	$y_1,y_2,\ldots, y_n$ denote normalized eigenvectors of $T^*$ 
	corresponding to the eigenvalues $\overline{\lambda_i}$.  
	Observe that $y_j$ is characterized up to  a scalar multiple by the fact that $\inner{x_i,y_j} = 0$ when $i \ne j$.
	Under these circumstances, it is known that the condition 
	\begin{equation}\label{eq-WAT}
		| \inner{ x_i, x_j}| = | \inner{ y_i, y_j} |
	\end{equation}
	for $1 \leq i < j \leq n$ is necessary for $T$ to be UECSM \cite[Thm.~1]{UECSMGC}
	(in fact, the first use of such a procedure in this context dates back to \cite[Ex.~7]{CSOA}).
	
	Although it was initially unclear whether \eqref{eq-WAT} is sufficient for $T$ to be UECSM,
	L.~Balayan and the first author eventually showed that there exist matrices
	$4 \times 4$ and larger which satisfy \eqref{eq-WAT} but which are not UECSM.  
	These counterexamples will be discussed further in Section \ref{SectionInsuffWAT}.
	On the other hand, based upon extensive numerical evidence they also conjectured
	that \eqref{eq-WAT} \emph{is} sufficient in the $3 \times 3$ case \cite[Sec.~6]{UECSMGC}.
	Theorem \ref{TheoremWAT} below settles this conjecture in the affirmative.
	
	Strangely enough, the proof relies critically upon complex function theory 
	and the emerging theory of truncated Toeplitz operators.
	Interest in truncated Toeplitz operators has blossomed over the last several years
	\cite{BCFMT, MR2597679, TTOSIUES, CRW, NLEPHS, Sarason, MR2468883, MR2418122, Sed, STZ},
	sparked by a seminal paper of D.~Sarason \cite{Sarason}.
	In \cite{ATTO}, W.T.~Ross and the first two authors established that 
	if $T \in {\bf M}_3(\C)$ has distinct eigenvalues $\lambda_1,\lambda_2,\lambda_3$ with corresponding normalized eigenvectors 
	$x_1, x_2, x_3$ satisfying $\inner{x_i,x_j} \neq 0$
	for $1 \leq i,j \leq 3$, then the following are equivalent:
	\begin{enumerate}\addtolength{\itemsep}{0.5\baselineskip}
		\item $T$ is unitarily equivalent to a complex symmetric matrix,
		\item $T$ is unitarily equivalent to an analytic truncated Toeplitz operator,
		\item The condition
			\begin{equation}\label{eq-Determinant}
				\det X^*X = (1 - | \inner{x_1 , x_2 }|^2)
				(1 - | \inner{x_2 , x_3 }|^2)(1 - | \inner{x_3 , x_1 }|^2)
			\end{equation}
			holds, where $X = ( x_1 | x_2 | x_3)$ is the matrix having
			$x_1, x_2, x_3$ as its columns.
	\end{enumerate}
	\smallskip
	In particular, a direct proof that (3) $\Rightarrow$ (1), independent of the theory of truncated
	Toeplitz operators, has not yet been discovered.

	\begin{Theorem}\label{TheoremWAT}
		Suppose that $T \in \M_3(\C)$ has distinct eigenvalues $\lambda_1, \lambda_2, \lambda_3$
		with corresponding unit eigenvectors $x_1, x_2, x_3$.  Let $y_1, y_2, y_3$
		denote unit eigenvectors of $T^*$ corresponding to the eigenvalues 
		$\overline{\lambda_1}, \overline{\lambda_2}, \overline{\lambda_3}$.  Under these circumstances,
		the condition \eqref{eq-WAT} is necessary and sufficient for $T$ to be UECSM.
	\end{Theorem}

\begin{proof}
	Since the necessity of \eqref{eq-WAT} is well-known \cite[Thm.~1]{UECSMGC}, we focus here on sufficiency.
	We first show that it suffices to consider the case where $\inner{x_i,x_j} \neq 0$ for $1 \leq i,j \leq 3$.
	
	Suppose that $T$ has a pair of eigenvectors which are orthogonal.  Upon scaling, translating
	by a multiple of the identity, and applying Schur's Theorem on unitary triangularization, we may further assume that
	\begin{equation*}
		T \cong \megamatrix{0}{0}{0}{a}{1}{0}{b}{0}{\lambda}
	\end{equation*}
	where $\lambda \neq 0,1$.  Since $T$ satisfies \eqref{eq-WAT}, the eigenspaces of $T^*$ corresponding to the eigenvalues $1$ and $\overline{\lambda}$ must be orthogonal.  A routine calculation shows that $(\overline{a},1,0)$ and $(\overline{b},0,\overline{\lambda})$ are eigenvectors of $T^*$ with eigenvalues $1$ and $\overline{\lambda}$, respectively, and so we must have $a=0$ or $b=0$.  It is straightforward to check that $T$ satisfies \eqref{eq-TraceTest} in either case, and thus $T$ is UECSM (one could also observe that 
both cases lead to the conclusion that $T$ is unitarily equivalent to the direct sum of a 
$2 \times 2$ and a $1 \times 1$ matrix whence $T$ is UECSM by any of 
\cite[Cor.~3]{UECSMGC}, \cite[Cor.~3.3]{Chevrot}, \cite[Ex.~6]{CSOA}, \cite{UET},
\cite[Cor.~1]{SNCSO}, \cite[p.~477]{McIntosh}, \cite[Cor.~3]{Tener}, or \cite[Ex.~2]{UECSMMC}).
	
	Assuming now that $\inner{x_i,x_j} \neq 0$ for $1 \leq i,j \leq 3$, we intend to use the fact
	that \eqref{eq-Determinant} implies that $T$ is UECSM.
	Let $X= ( x_1 | x_2 | x_3)$ and $Y = ( y_1 | y_2 | y_3)$
	denote the $3 \times 3$ matrices having the vectors $x_1, x_2, x_3$ 
	and $y_1, y_2, y_3$ as columns, respectively.  In particular, note that
	\begin{equation}\label{eq-Duality}
		Y^*X = \megamatrix{ \inner{ x_1, y_1} }{0}{0}{0}
		{ \inner{ x_2, y_2} }{0}{0}{0}{\inner{x_3, y_3}}.
	\end{equation}
	
	We now claim that
	\begin{equation}\label{eq-ComputationalClaim}
		\det X^*X = | \inner{ x_1, y_1} |^2 ( 1 - | \inner{ x_2, x_3} |^2).
	\end{equation}
	Since $y_1$ is a unit vector orthogonal to $x_2$ and $x_3$, we may write 
	\begin{equation*}
		x_1 = \inner{x_1,y_1}y_1 + x^\prime
	\end{equation*}
	for some $x^\prime$ in $\operatorname{span} \{x_2, x_3\}$.  Let $\Lambda$ be the multilinear function given by
	\begin{equation*}
		\Lambda(w_1,w_2,w_3) = \det X^*W,
	\end{equation*}
	where $W=(w_1|w_2|w_3)$ is the matrix whose columns are the $w_i$.  We then have
	\begin{align*}
		\det X^*X &= \Lambda(x_1,x_2,x_3)\\
		&= \inner{x_1,y_1} \Lambda(y_1,x_2,x_3) + \Lambda(x^\prime,x_2,x_3)
	\end{align*}
	Since $x^\prime$ belongs to $\operatorname{span} \{x_2,x_3\}$, the second term vanishes and we have
	\begin{equation*}
		\det X^*X = \inner{x_1,y_1} \det \megamatrix{\inner{y_1,x_1}}{\inner{x_2,x_1}}{\inner{x_3,x_1}}{0}{1}{\inner{x_3,x_2}}{0}{\inner{x_2,x_3}}{1},
	\end{equation*}
	from which the desired condition \eqref{eq-ComputationalClaim} is immediate.

	Similarly we obtain
	\begin{align}
		\det X^*X &=  |\inner{x_2,  y_2}|^2(1 - | \inner{ x_3, x_1} |^2) ,\label{eq-CS2}\\
		\det X^*X &=  |\inner{x_3,  y_3}|^2(1 - | \inner{ x_1, x_2} |^2),\label{eq-CS3}
	\end{align}
	by relabeling the indices and using the same argument.  Moreover,
	we can also perform these computations with $Y^*Y$ in place of $X^*X$, which provides
	\begin{equation}\label{eq-ComputationalClaim2}
		\det Y^*Y = |\inner{x_1,  y_1}|^2(1 - | \inner{ y_2, y_3} |^2).
	\end{equation}
	
	Thus if $T$ satisfies \eqref{eq-WAT}, then it follows from 
	\eqref{eq-ComputationalClaim} and \eqref{eq-ComputationalClaim2} that
	\begin{equation*}
		|\det X|^2 = \det X^*X = \det Y^*Y = | \det Y|^2,
	\end{equation*}
	whence $|\det X| = |\det Y|$.
	Multiplying \eqref{eq-ComputationalClaim}, \eqref{eq-CS2}, and \eqref{eq-CS3} together and appealing to \eqref{eq-Duality},
	we obtain
	\begin{equation*}
		(\det X^*X)^3 = |\det Y^*X|^2 (1 - | \inner{ x_1, x_2} |^2)(1 - | \inner{ x_2, x_3} |^2)
		(1 - | \inner{ x_3, x_1} |^2).
	\end{equation*}
	However, since $\det X^*X = |\det X|^2 = |\det Y||\det X| = |\det Y^*X|$ it follows from the preceding that
	\begin{equation*}
		\det X^*X = (1 - | \inner{ x_1, x_2} |^2)(1 - | \inner{ x_2, x_3} |^2)
		(1 - | \inner{ x_3, x_1} |^2).
	\end{equation*}
	As we have discussed above, this establishes that $T$ is UECSM.
\end{proof}

\section{The angle criterion in dimensions $n \geq 4$}\label{SectionInsuffWAT}
	Following the notation and conventions established in Section \ref{SectionWAT}, 
	we assume that the matrix $T$ in $\M_n(\C)$ has distinct eigenvalues 
	$\lambda_1, \lambda_2,\ldots,\lambda_n$ and corresponding normalized
	eigenvectors $x_1,x_2,\ldots,x_n$.  Similarly, we select normalized eigenvectors of $T^*$ corresponding
	to the eigenvalues $\overline{\lambda_1}, \overline{\lambda_2}, \ldots, \overline{\lambda_n}$
	and denote them $y_1,y_2,\ldots,y_n$.  Recall from the preceding discussion 
	that the condition
	\begin{equation}\tag{\ref*{eq-WAT}}
		| \inner{ x_i, x_j}| = | \inner{ y_i, y_j} |
	\end{equation}
	for $1 \leq i < j \leq n$ is necessary and sufficient for $T$ to be UECSM if $n \leq 3$,
	but insufficient if $n \geq 4$.  Indeed, there exist matrices $4 \times 4$ or larger which are \emph{not} UECSM
	but which nevertheless satisfy \eqref{eq-WAT}.  The first known example
	was discovered by L.~Balayan using a random search of matrices having integer 
	entries \cite[Ex.~5]{UECSMGC}.  In this section, we provide a solid theoretical explanation for 
	the existence of such counterexamples and we illustrate this process by constructing 
	L.~Balayan's original counterexample from scratch.
	
	Unlike \eqref{eq-WAT}, the related condition
	\begin{equation}\label{eq-SAT}
	 	\inner{x_i,x_j}\inner{x_j,x_k}\inner{x_k,x_i} = \overline{\inner{y_i,y_j}\inner{y_j,y_k}\inner{y_k,y_i}},
	\end{equation}
	for $1 \le i \le j \le k \le n$, is equivalent to asserting that $T$ is UECSM \cite[Thm.~2]{UECSMGC}.  
	Following \cite{UECSMGC}, we refer to \eqref{eq-SAT} as the \emph{Strong Angle Test} (SAT) 
	and \eqref{eq-WAT} as the \emph{Weak Angle Test} (WAT).   Observe that the WAT can be obtained
	from the SAT by setting $k=j$ in \eqref{eq-SAT}.  In particular, we remark that a matrix which
	passes the SAT automatically passes the WAT, although the converse does not hold.
	 
	Curiously, the counterexample discussed above satisfies the related condition
	\begin{equation}\label{eq-LSAT}
	 	\inner{x_i,x_j}\inner{x_j,x_k}\inner{x_k,x_i} =\inner{y_i,y_j}\inner{y_j,y_k}\inner{y_k,y_i},
	\end{equation}
	for $1 \le i \le j \le k \le n$ \cite[Ex.~5]{UECSMGC}.   
	We say that a matrix which satisfies \eqref{eq-LSAT} passes the \emph{Linear Strong
	Angle Test} (LSAT).   Our aim in this section is to describe a method 
	for producing matrices which pass the LSAT \eqref{eq-LSAT}
	and hence the WAT \eqref{eq-WAT}, but not the SAT \eqref{eq-SAT}.

	\begin{Theorem}\label{TheoremLSAT}
		A matrix $T$ in $\M_n(\C)$ which has distinct eigenvalues satisfies the Linear
		Strong Angle Test \eqref{eq-LSAT} 
		if and only if $T$ is unitarily equivalent to a matrix of the form
		$QDQ^{-1}$ where $D$ is diagonal and $Q$ belongs to $\SU(k,n-k)$ for some $1 \leq k \leq n$.
	\end{Theorem}

	Here $\SU(k,n-k)$ refers to the group of complex matrices having 
	determinant $1$ and which preserve the Hermitian form
	\begin{equation*}
		\inner{v,w}_k := \sum_{j=1}^k v_j\overline{w_j} ~- \sum_{j=k+1}^n v_j \overline{w_j}.
	\end{equation*}
	In particular, we observe that a matrix $Q$ belongs to $\SU(k,n-k)$ if and only if 
	\begin{equation}\label{eq-QAQA}
		Q^*AQ = A,
	\end{equation}
	where
	\begin{equation*}
		A:= I_k \oplus -I_{n-k}.
	\end{equation*}	
	In contrast, a matrix passes the Strong Angle Test \eqref{eq-SAT} if and only if it is unitarily equivalent to a matrix 
	of the form $QDQ^{-1}$ where $D$ is diagonal and 
	$Q$ belongs to $\O(n)$, the complex orthogonal group of order $n$ \cite[Thm.~4.4.13]{HJ}
	(see also \cite[Sect.~5]{ESCSO}).

	In order to prove Theorem \ref{TheoremLSAT}, we require the following lemma.

	\begin{Lemma}\label{LemmaLSATUnitary}
		Maintaining the notation and conventions established above,
		if a matrix $T$ in $\M_n(\C)$ 
		satisfies the Linear Strong Angle Test \eqref{eq-LSAT}, then there is a selfadjoint unitary matrix 
		$U$ and unimodular constants $\alpha_1,\alpha_2,\ldots,\alpha_n$
		such that $Ux_i = \alpha_i y_i$ for $1 \le i \le n$.
	\end{Lemma}
	
	\begin{proof}[Proof of Lemma \ref{LemmaLSATUnitary}]
		The proof is similar to that of \cite[Thm.~2]{UECSMGC}.
		We first work under the assumption that $\inner{x_i,x_j} \ne 0$ for all 
		$i,j$.  Setting $k = i$ in \eqref{eq-LSAT} reveals that
		$\left| \inner{y_i,y_j} \right| =  \left| \inner{x_i,x_j} \right|$ so that the constants
		\begin{equation*}
			\gamma_i := \frac{\inner{y_1,y_i}}{\inner{x_1,x_i}}
		\end{equation*}
		each have unit modulus.  Since $T$ satisfies the LSAT, we next observe that 
		\begin{equation}\label{eq-alphaprop}
			\gamma_i \overline{\gamma_j} = \frac{\inner{y_1,y_i}\inner{y_j,y_1}}{\inner{x_1,x_i}\inner{x_j,x_1}} 
			= \frac{\inner{x_i,x_j}}{\inner{y_i,y_j}}.
		\end{equation}
		Let $R$ be the $n \times n$ matrix which satisfies $Rx_i = \gamma_i y_i$ for $1 \leq i \leq n$.
		By \eqref{eq-alphaprop}, we see that
		\begin{equation*}
			\inner{Rx_i,Rx_j} = \gamma_i \overline{\gamma_j} \inner{y_i,y_j} = \inner{x_i,x_j}
		\end{equation*}
		from which it follows that $R$ is unitary.
		We now briefly sketch how to modify this construction if $\inner{x_i,x_j} = 0$ for some pair $(i,j)$.  
		The details are largely technical and can be found in the proof of 
		\cite[Thm.~2]{UECSMGC}, \emph{mutatis mutandis}.  
	
		Consider the partially-defined, selfadjoint matrix $\left(\beta_{ij}\right)_{i,j=1}^n$ 
		whose (obviously unimodular) entries are given by 
		\begin{equation*}
			\beta_{ij} = \frac{\inner{y_i,y_j}}{\inner{x_i,x_j}}, 
		\end{equation*}
		for those $1 \le i,j \le n$ for which this expression is well-defined. 
		Since $T$ satisfies the LSAT \eqref{eq-LSAT}, it follows that $\beta_{ij}\beta_{jk}=\beta_{ik}$ holds
		whenever all of the quantities involved are well-defined.  It turns out that one can inductively fill in the undefined 
		entries of the matrix $\left(\beta_{ij}\right)_{i,j=1}^n$ so that each entry $\beta_{ij}$ is unimodular
		(i.e., $\beta_{ij} = \overline{\beta_{ji}}$) and such that the multiplicative property $\beta_{ij}\beta_{jk}=\beta_{ik}$ holds
		whenever $1 \leq i,j,k\leq n$.  One then constructs the unitary matrix $R$ by setting 
		$\gamma_i = \beta_{1i}$ and letting $Rx_i = \gamma_i y_i$ as before.  We refer the reader to
		the proof of \cite[Thm.~2]{UECSMGC} for further details.
		
		Now let $X = (x_1 | x_2| \cdots | x_n )$ denote the $n \times n$ matrix whose columns
		are the eigenvectors $x_1,x_2,\ldots,x_n$ of $T$ and let $D = \diag(\lambda_1,\lambda_2,\ldots,\lambda_n)$
		be the diagonal matrix whose entries are the corresponding eigenvalues of $T$.
		In particular, we note that the matrix
		$X$ diagonalizes $T$ in the sense that
		\begin{equation}\label{eq-Diagonalize}
			T = XDX^{-1}.
		\end{equation}
		We next remark that
		\begin{equation}\label{eq-UEtoop1}
			R^*T^*R = XD^*X^{-1},  
		\end{equation}
		since both matrices agree on the basis $x_1,x_2,\ldots,x_n$.
		Taking adjoints in \eqref{eq-Diagonalize} we find that
		\begin{equation*}
			T^* = (X^*)^{-1}D^*X^*,
		\end{equation*}
		from which it follows that the $i$th column of $(X^*)^{-1}$ is an eigenvector of $T^*$
		corresponding to the eigenvalue $\overline{ \lambda_i}$.  One can therefore check that
		\begin{equation}\label{eq-UEtoop2}
			RTR^* = (X^*)^{-1}DX^*
		\end{equation}
		by noting that $(X^*)^{-1}DX^*$ has the same eigenvectors as $T^*$ and
		evaluating both sides of \eqref{eq-UEtoop2} on the basis $y_1,y_2,\ldots,y_n$.  
		Taking adjoints in \eqref{eq-UEtoop2} yields 
		\begin{equation*}
			RT^*R^* = XD^*X^{-1}.
		\end{equation*} 
		Comparing the preceding with \eqref{eq-UEtoop1} we find that 
		\begin{equation*}
			RT^*R^* = R^*T^*R.
		\end{equation*}
		In other words, $T^*$ commutes with the unitary matrix $R^2$.  	
		
		If $T$ is irreducible (i.e., has no proper, nontrivial reducing subspaces), then 
		$R^2 = \omega I$ for some constant $\omega$ of unit modulus.  Letting $\alpha_i = \omega^{-1/2} \gamma_i$
		(either branch of the square root is acceptable),
		we find that the matrix $U = \omega^{-1/2} R$ is selfadjoint, unitary, and satisfies
		$Ux_i = \alpha_i y_i$ for $1 \leq i \leq n$.
		To conclude the proof in the general case, one simply applies the preceding reasoning on each 
		maximal proper reducing subspace of $T$ to obtain the desired matrix $U$.
	\end{proof}
	
	\begin{proof}[Proof of Theorem \ref{TheoremLSAT}]
		$(\Leftarrow)$ First assume that $T$ is unitarily equivalent to $QDQ^{-1}$ where $D$ diagonal and $Q$
		belongs to $\SU(k,n-k)$.  Since the condition \eqref{eq-LSAT} of the LSAT is invariant under unitary transformations, 
		we may assume that $T = QDQ^{-1}$.  Recalling that $A = I_k \oplus (-I_{n-k})$ is diagonal
		and using \eqref{eq-QAQA}, we have
		\begin{align*}
			T^* &= (Q^{-1})^*D^*Q^*\\
			&= (AQA)D^* (AQ^{-1}A) \\
			&= AQ(AD^* A)Q^{-1}A \\
			&= (AQ)D^* (AQ)^{-1}.
		\end{align*}
		Writing $Q = (q_1 |q_2| \cdots | q_n)$ in column-by-column format, we obtain normalized eigenvectors
		\begin{equation*}
			x_i = \frac{q_i}{\norm{q_i}}, \qquad y_i = Ax_i,
		\end{equation*}
		of $T$ and $T^*$, respectively.  Since $A$ is unitary it follows that $\inner{x_i,x_j} = \inner{y_i,y_j}$ 
		whence $T$ satisfies the Linear Strong Angle Test \eqref{eq-LSAT}.
		\medskip
		
		\noindent$(\Rightarrow)$  Suppose that $T$ satisfies the LSAT. 	
		By Lemma \ref{LemmaLSATUnitary}, there is a selfadjoint unitary matrix $U$ and unimodular
		constants $\alpha_1, \alpha_2, \ldots, \alpha_n$ such that $Ux_i = \alpha_i y_i$ for $1 \le i \le n$.  
		As in the proof of Lemma \ref{LemmaLSATUnitary}, let $X = (x_1 | x_2 | \cdots | x_n)$ and let 
		$D = \diag(\lambda_1, \lambda_2, \ldots, \lambda_n)$ so that 
		\begin{equation}\label{eq-TXDX}
			T = XDX^{-1}.  
		\end{equation}
		As in \eqref{eq-UEtoop2}, we have 
		\begin{equation}\label{eq-TXDX2}
			T = U(X^*)^{-1}DX^*U,
		\end{equation}
		as both sides agree on the basis $x_1,x_2,\ldots,x_n$ (recall that $(X^*)^{-1}DX^*$ has the same
		eigenvectors as $T^*$).  In light of \eqref{eq-TXDX} and \eqref{eq-TXDX2}, we conclude that
		the matrix $X^*UX$ commutes with $D$.  Since the diagonal entries of $D$ are distinct
		and $U$ is selfadjoint, we conclude that $X^*UX$ is a diagonal matrix having only real entries.
		 Without loss of generality, we may assume that the vectors $x_1,x_2,\ldots,x_n$ are ordered 
		 so that the first $k$ diagonal entries of $X^*UX$ are positive and the last $n-k$ are negative
		 (note that $X^*UX$ is invertible since both $U$ and $X$ are invertible).
	 	 
		Now let $w_i = \left|\delta_i\right|^{-\frac12} x_i$, where $\delta_i$ is the $i$th diagonal entry of $X^*UX$.  
		With $W = (w_1 | w_2|\cdots | w_n)$, we have
		\begin{equation}\label{eq-WUWB}
			W^*UW=A.
		\end{equation}
		Since $U$ is selfadjoint and unitary, we may appeal to both Sylvester's Law of Inertia \cite[Thm.~4.5.8]{HJ} 
		and the Spectral Theorem to find a unitary matrix $Z$ such that 
		\begin{equation}\label{eq-UZAZ}
			U=Z^*AZ.
		\end{equation}
		Plugging \eqref{eq-UZAZ} into \eqref{eq-WUWB} we find that
		\begin{equation}\label{eq-ZWAZWA}
			(ZW)^*A(ZW) = A,
		\end{equation}
		which tells us that the matrix $Q= ZW$ belongs to $\SU(k,n-k)$.
		Since the columns $w_i$ of $W$ are nonzero multiples of the corresponding columns
		$x_i$ of $X$, it follows that
		\begin{align*}
			T 
			&= WDW^{-1} \\
			&= Z^*\left[(ZW)D(ZW)^{-1}\right]Z \\
			&= Z^* (QDQ^{-1}) Z.
		\end{align*}
		Thus $T$ is unitarily equivalent to a matrix of the form $QDQ^{-1}$ where
		$Q$ belongs to $\SU(k,n-k)$.  This completes the proof of Theorem \ref{TheoremLSAT}.
	\end{proof}
	
	We can now use Theorem \ref{TheoremLSAT}  to construct matrices that satisfy the Weak 
	Angle Test but not the Strong Angle Test.  We begin by constructing a matrix $T$ that 
	satisfies the Linear Strong Angle Test.  From the theorem, we know that this can be done 
	by constructing a matrix $Q$ in $\SU(k,n-k)$ and setting $T = QDQ^{-1}$ for any diagonal
	matrix $D$ with distinct entries.  Putting $k=j$ in the LSAT shows that $T$ will satisfy the 
	WAT, but $T$ may satisfy the SAT as well.  
	
	Comparing \eqref{eq-SAT} and \eqref{eq-LSAT}, we can see that $T$ will satisfy the SAT if and only if
	\begin{equation}\label{eq-RealInners}
		\inner{x_i,x_j} \inner{x_j,x_k} \inner{x_k, x_i} \in \R
	\end{equation}
	for $1 \le i \le j \le k \le n$.  In practice, this condition is rarely satisfied.  We illustrate the
	 process with the following example.

	\begin{Example}\label{ExampleConstruct}
		To construct a matrix that satisfies the WAT but not the SAT, we first need to construct an element 
		of $\SU(k,n-k)$.  We will do this for $n=4$, as we know from Theorem \ref{TheoremWAT} that 
		examples of matrices that satisfy WAT but not SAT do not exist for smaller choices of $n$.  We will 
		use $k=2$, which turns out to be necessary when $n=4$ (see Theorem \ref{TheoremSimple}).  
		
		An element of $\SU(2,2)$ can be
		produced by applying an indefinite analogue of the Gram-Schmidt process
		to a collection of four elements of $\C^4$ \cite[Sec. 3.1]{MR2186302}.  Using this method on a
		 matrix with small random entries in $\mathbb{Z}[i]$ produced
		\begin{equation*}
			Q = 
			\begin{pmatrix}
				 1+\frac{i}{2} & 0 & -\frac{1}{2\sqrt{6}} (1-i) & \frac{i}{\sqrt{6}} \\
				 -\frac{i}{2} & 2 i & \frac{1}{2\sqrt{6}}(7+5i) &   -\frac{i}{\sqrt{6}} \\
				 -\frac{1}{2}(1-i) & 1-i & -\frac{1}{\sqrt{6}}(1+4 i) & -\sqrt{\frac{2}{3}} \\
				 0 & -i & -\sqrt{\frac{2}{3}}(1+i)  & \sqrt{\frac{2}{3}}
			\end{pmatrix}.
		\end{equation*}
		
		Let $D$ be the diagonal matrix with diagonal $(-1, 0, 1,2)$ and let $T = QDQ^{-1}$.  Explicitly, we have
		\begin{equation}\label{eq-JamesMatrix}
			T = \frac{1}{6}
			\begin{pmatrix}
				-10 & 4-6 i & -3-11 i & 2 i \\
				 4+6 i & -22 & -15+17 i & -12-2 i \\
				 3-11 i & 15+17 i & 28 & 2+6 i \\
				 2 i & 12-2 i & 2-6 i & 16
			\end{pmatrix}.
		\end{equation}
		By Theorem \ref{TheoremLSAT}, we know that $T$ passes the LSAT and hence the WAT.  
		On the other hand, if $q_i$ denotes the $i$th column of $Q$, then we have
		\begin{equation*}
			\inner{q_1,q_2}\inner{q_2,q_3}\inner{q_3,q_1} = \frac{1}{3}(100 - 8i) \not \in \R.
		\end{equation*}
		Hence $T$ does not satisfy the SAT and is therefore not UECSM.  Although there was no guarantee that the matrix $T$ 
		obtained in this manner would not satisfy the SAT, in practice this does not appear to occur frequently.
	\end{Example}
		
	\begin{Example}\label{ExampleLevon}
		In this example, we consider the matrix 
		\begin{equation}\label{eq-Levon}
			T=
			\begin{pmatrix}
				 5 & 0 & -1 & 3 \\
				 2 & 4 & 1 & 2 \\
				 2 & -2 & 6 & -2 \\
				 0 & -2 & 1 & 4
			\end{pmatrix},
		\end{equation}
		which was the first known example of a matrix which passes the Weak Angle Test \eqref{eq-WAT}
		yet fails to be UECSM \cite[Ex.~5]{UECSMGC}.  This matrix was originally obtained by a brute force search, 
		but now Theorem \ref{TheoremLSAT} puts it into a broader context  and explains \emph{why} matrices such 
		as \eqref{eq-Levon} exist.  Indeed, the computations carried out in
		\cite[Ex.~5]{UECSMGC} confirm that $T$ passes the WAT and the LSAT, but fails
		the SAT.  By following the proof of Theorem \ref{TheoremLSAT}, it is possible to explicitly compute a $Q$ in
		$\SU(2,2)$ and a diagonal matrix $D$ such that $T$ is unitarily equivalent to $QDQ^{-1}$.
	\end{Example}
	
\section{Contrasting dimensions three and four}\label{Section34}

	We conclude this note with some remarks concerning certain phenomena which
	distinguish dimension three  from dimensions four and above.
	In the following, we maintain the notation and conventions established in the preceding two sections.
	
	As we have seen, Theorem \ref{TheoremLSAT} provides a method for constructing
	matrices which pass the Weak Angle Test (WAT) and which \emph{may} fail to be UECSM.  
	On the other hand, Theorem \ref{TheoremWAT} asserts that  passing the WAT \emph{is} sufficient for a matrix to be UECSM
	if $n = 3$.  Therefore something peculiar must occur in dimension three
	which prevents the method of Theorem \ref{TheoremLSAT} from ever actually producing
	examples such as the matrices \eqref{eq-JamesMatrix} from Example \ref{ExampleConstruct} and
	 \eqref{eq-Levon} from Example \ref{ExampleLevon}.
	The following theorem helps explain this curious dichotomy.  
	
	\begin{Theorem}\label{TheoremSimple}
		If $T$ in $\M_n(\C)$ has distinct eigenvalues and is unitarily equivalent to a matrix of the form 
		$QDQ^{-1}$ where $D$ is diagonal and $Q$ belongs to $\SU(n-1,1)$,
		then $T$ is UECSM.
	\end{Theorem}

	\begin{proof}
		Without loss of generality, we may assume that $T = QDQ^{-1}$
		where $D$ is diagonal and $Q$ belongs to $\SU(n-1,1)$.  We may also assume 
		that $D$ has real entries, as it follows from the Strong Angle Test that $T$ being 
		UECSM is independent of the actual eigen\emph{values} of $T$.  By Theorem 
		\ref{TheoremLSAT}, it follows that $T$ satisfies the 
		Linear Strong Angle Test \eqref{eq-LSAT}. 
		Returning to the proof of Theorem \ref{TheoremLSAT}, we note that \eqref{eq-TXDX2} asserts that
		\begin{equation*}
			T = U(X^*)^{-1}DX^*U
		\end{equation*}
		where $U$ is a selfadjoint unitary matrix and $D= D^*$ by assumption.
		However, this simply means that 
		\begin{equation}\label{eq-TUTU}
			T = UT^*U.
		\end{equation}
		As shown in the proof of Theorem \ref{TheoremLSAT},
		the fact that $Q$ belongs to $\SU(n-1,1)$ implies that $U$ is unitarily equivalent to
		$A = \diag(1,1,\ldots,1,-1)$.  Indeed, plugging $Q = ZW$ into \eqref{eq-ZWAZWA}
		and using \eqref{eq-UZAZ}, we find that $W^*UW = A$ (i.e., $U$ is $*$-congruent to $A$).  
		The desired result follows upon appealing to Sylvester's Law of Inertia \cite[Thm.~4.5.8]{HJ}.
		
		After performing a unitary change of coordinates in \eqref{eq-TUTU},
		we may assume that $U = \diag(1,1,\ldots,1,-1)$ so that $T$ has the form
		\begin{equation*}
			\begin{pmatrix}
				T_{1,1} & T_{1,2}\\
				-T_{1,2}^* & T_{2,2}
			\end{pmatrix}
		\end{equation*}
		where $T_{1,1}$ is $(n-1)\times (n-1)$ and selfadjoint and $T_{2,2}$ is $1\times 1$ and real.
		Conjugating $T$ by an appropriate block-diagonal unitary matrix we may further assume that 
		$T_{1,1}$ is diagonal.  Conjugating again by a diagonal unitary matrix, we may also
		arrange for the $(n-1)\times 1$ matrix $T_{1,2}$ to be purely imaginary.  In other words,
		$T$ is UECSM.
	\end{proof}

	\begin{Corollary}\label{Corollary3x3}
		If $T$ is $3 \times 3$ and satisfies the Linear Strong Angle Test \eqref{eq-LSAT}, then $T$ is UECSM.
	\end{Corollary}
	
	\begin{proof}
		By Theorem \ref{TheoremLSAT}, $T$ is unitarily equivalent to a matrix of the form $QDQ^{-1}$ where
		$Q$ belongs to either $\SU(3,0)$ or $\SU(2,1)$ (the cases $\SU(1,2)$ and $\SU(0,3)$ being identical
		to these first two).  If $Q$ belongs to $\SU(3,0)$, then $Q$ is unitary whence $T$ is unitarily equivalent to
		the diagonal matrix $D$.  On the other hand, if $Q$ belongs to $\SU(2,1)$, then we may appeal to
		Theorem \ref{TheoremSimple} to conclude that $T$ is UECSM.
	\end{proof}

%


\bibliography{UECSMLD}

\end{document}